\documentclass[12pt]{amsart}
\usepackage{amsfonts, amsmath, amsthm, amssymb, stmaryrd, dsfont}
\usepackage{textcomp}
\usepackage{bbm}
\usepackage{enumitem, tabto}
\usepackage{xcolor}
\usepackage{url}
\usepackage{graphicx}
\usepackage{fancybox}
\usepackage{awesomebox}
\usepackage{tikz}
\usepackage{pgfplots}
\usepackage[Glenn]{fncychap}
\usepackage{fancyhdr}
\usepackage{lastpage}
\usepackage{import}

\usepackage{diagbox}
\usepackage{multirow} 
\usepackage{tabularx}

\usepackage{etoolbox}

\usepackage{hyperref}

\usepackage{blindtext}
\usepackage{multicol}
\usepackage{wrapfig}

\usepackage{tkz-tab}

\usepackage{pdfpages}

\usepackage{comment}

\newcommand{\C}{\mathbb{C}}
\newcommand{\D}{\mathbb{D}}

\newcommand{\T}{\mathbb{T}}
\newcommand{\Lcont}{\mathcal{L}}

\newcommand{\norm}[1]{\left\Vert #1\right\Vert}
\newcommand{\abs}[1]{\left\lvert #1 \right\rvert}
\newcommand{\scal}[2]{\ensuremath{\left\langle #1|#2 \right\rangle}\xspace}

\newcommand{\Span}{\operatorname{span}}

\usepackage{indentfirst}
\makeatletter
\@addtoreset{equation}{section}
\makeatother
\theoremstyle{plain}
\newtheorem{statement}{}[section]
\newtheorem{theo}[statement]{Theorem}

\newtheorem{coro}[statement]{Corollary}
\newtheorem{lem}[statement]{Lemma}

\theoremstyle{definition}

\newtheorem{rem}[statement]{Remark}
\newtheorem{expl}[statement]{Example}

\title[Cyclicity of the shift and a completeness problem]{Cyclicity of the shift operator and a related completeness problem in De Branges-Rovnyak spaces
}
\author{Emmanuel Fricain and Romain Lebreton}
\address{Laboratoire Paul Painlevé, Université de Lille, 59655 Villeneuve d'Ascq Cédex}
\email{emmanuel.fricain@univ-lille.fr}
\email{romain.lebreton@univ-lille.fr}

\thanks{The authors were supported by the Labex CEMPI (ANR-11-LABX-0007-01)}

\begin{document}

\begin{abstract}
In this paper, we study the cyclic vectors of the shift operator $S_b$ acting on de Branges--Rovnyak space $\mathcal H(b)$ associated to a non-extreme point of the closed unit ball of $H^\infty$. We highlight an interesting link with a completeness problem that we study using the Cauchy transform. This enables us to obtain some nice consequences on cyclicity. 
\end{abstract}

\maketitle

\section{Introduction}
If $T$ is a bounded linear operator on a Banach space $\mathcal X$, then $T$ is called {\emph{cyclic}} if there exists a vector $x\in\mathcal X$ such that the orbit of $x$ under $T$, that is $\{T^n x:n\geq 0\}$, is dense in $\mathcal X$. Such a vector (if it exists) is called a {\emph{cyclic vector}} for $T$. The characterization of cyclic vectors for a given operator is a challenging problem. It has been completely solved by A. Beurling \cite{Beurling} for the (forward) shift operator $S$ on the Hardy space of the unit disc $H^2$. A function $f$ in $H^2$ is cyclic for $S$ if and only if $f$ is an outer function. A similar question can be stated (and has been studied) in various Banach spaces of analytic functions where the shift operator acts boundedly, e.g. in Bergman or Dirichlet spaces. However, the situation in Hardy space is unique in the sense that in most other spaces, there are no known characterizations despite numerous efforts by many mathematicians. In \cite{BrownShields}, L. Brown and A. Shields conjectured that a function $f$ in the Dirichlet space $\mathcal D$ is cyclic if and only if $f$ is outer and its boundary zero set is of logarithmic capacity zero. This conjecture is still open despite significant progress. See \cite{CyclDirichlet} in particular. 

Recently, some authors began to be interested in cyclic vectors of the shift operator on de Branges--Rovnyak spaces $\mathcal H(b)$ \cite{InvSubHbForm,CyclBG,CyclNonExtr}. Recall that when $\|b\|_\infty\leq 1$ and $\log(1-|b|)\in L^1(\mathbb T)$ (corresponding to the non-extreme point of the closed unit ball of $H^\infty$), the space $\mathcal H(b)$ is invariant with respect to the shift operator $S_b$ and since the polynomials are dense in this case, the operator $S_b$ is cyclic (with the constant $1$ as a cyclic vector). It is then natural to ask if one can characterize all the cyclic vectors of $S_b$. In \cite{CyclFG}, the first author with S. Grivaux gave a characterization of the cyclic vectors of $S_b$ when $b$ is a rational function in the closed unit ball of $H^\infty$ but not a finite Blaschke product, which generalizes some results of \cite{CyclNonExtr}. They also gave some sufficient conditions when $b=(1+I)/2$, where $I$ is an inner function. In \cite{CyclBG}, A. Bergman gave deep results based on a theoretic description of the invariant subspaces given by A. Aleman and B. Malman \cite{InvSubHbForm}. In particular, he completely solved the problem when $b=(1+I)/2$, where $I$ is an inner function, with a nice characterization of cyclic vectors in terms of the Aleksandrov--Clark measures of $b$. 

In this paper, we shed new light to this problem of characterizing  the cyclic vectors of $S_b$ by making a link with the problem of completeness in $\mathcal H(b)$ of families of the form $\{fk_{\lambda_n}:n\geq 1\}$ where $f$ is a fixed function in $\mathcal H(b)$, $(\lambda_n)_{n\geq 1}$ is a sequence of points in $\mathbb D$ which is not a Blaschke sequence, and $k_{\lambda_n}$ is the Cauchy kernel of $H^2$ at point $\lambda_n$. A difficulty with de Branges--Rovnyak spaces is the computation of the norm because it is not given directly by an integral. In the non-extreme case, usually one need to solve the equation $T_{\bar b}f=T_{\bar a}f^+$, where $f$ is given in $\mathcal H(b)$ and $T_{\bar b}$ and $T_{\bar a}$ are the Toeplitz operators of symbols $\bar{b}$ and $\bar{a}$ respectively and $a$ is the Pythagorean mate of $b$. See Subsection 2.1 for more details. In particular, when one studies directly the cyclicity, we are faced with the difficulty of getting a tractable formula for $(pf)^+$ (where $p$ is a polynomial), and this is not an easy task! However, we shall see that we can obtain a nice formula for $(fk_{\lambda})^+$, which will help us, using Cauchy transform, to give a sufficient and simple condition for completeness, and then deduce some nice results for cyclicity.

The paper is organized as follows. In Section 2, we recall some useful properties of the de Branges--Rovnyak spaces and the Cauchy transform. We also give a key lemma for the study of our completeness problem. In Section 3, we introduce a completeness problem related to the cyclic vectors for the shift operator. First of all, this issue will be considered in the Hardy space $H^2$, and then in the de Branges--Rovnyak spaces $\mathcal{H}(b)$. In Section 4, we prove our main result which gives a sufficient condition to solve our completeness problem. Finally, the last section contains several consequences of our approach on the cyclic vectors for $S_b$.

In the rest of the paper, if $A$ is a family of vectors in a Hilbert space $\mathcal H$, we shall denote by $\bigvee(A)$ the subspace consisting of finite linear combinations of elements of $A$ and by $\Span_{\mathcal H}(A)$ its closure in $\mathcal H$.

\section{Preliminaries and useful analytic tools}\label{Preliminaries}

\subsection{De Branges-Rovnyak spaces \texorpdfstring{$\mathcal{H}(b)$}{Hb}}

Let $\D = \{z \in \C : \abs{z} < 1\}$ be the open unit disk in the complex plane and let $\T = \{z \in \C : \abs{z} = 1 \}$ its boundary endowed with normalized Lebesgue measure $m$. For $0<p<\infty$, we consider the Hardy space $H^p=H^p(\D)$ which consists of functions $f$ holomorphic on $\D$ satisfying
\[
\norm{f}_p= \sup_{0 \le r < 1}\left(\int_{\T} \abs{f(r\zeta)}^p\,dm(\zeta)\right)^{1/p}<\infty.
\]
We also define $H^\infty = H^\infty(\D)$ to be the class of bounded analytic functions on $\D$, endowed with the sup norm defined by $\norm{f}_\infty = \sup_{z \in \D} \abs{f(z)}$.  Throughout this paper, as usual, for every $1\leq p\leq \infty$, we identify (via radial limits) $H^p(\D)$ with the closed subspace of $L^p(\T)$ defined as $H^p(\T) := \{f \in L^p(\T) : \widehat{f}(n) = 0,\,\forall n < 0 \}$. Recall that for $p=2$, $H^2$ is a reproducing kernel Hilbert space whose kernel is given by 
\[
k_\lambda(z) = \frac{1}{1 - \overline{\lambda}z}, \quad \lambda, z \in \D,
\]
meaning that 
\[
\scal{f}{k_\lambda}_2=f(\lambda),\qquad f\in H^2,\lambda \in\D,
\]
and $\scal{\cdot}{\cdot}_2$ is the usual scalar product of $L^2(\T)$.

For $\varphi\in L^\infty(\T)$, the Toeplitz operator $T_\varphi:H^2\to H^2$, defined by $T_\varphi f = P_+(\varphi f)$ for every $f \in H^2$, is a bounded operator on $H^2$ of norm equals to $\|\varphi\|_\infty$. Here $P_+$ is the orthogonal projection from $L^2(\T)$ to $H^2$.

To every non-constant function $b$ in the closed unit ball of $H^\infty$, we associate the de Branges--Rovnyak space $\mathcal H(b)$ 
defined as the reproducing kernel Hilbert space on $\D$  with positive definite kernel given by 
\[
k_\lambda^b(z) = \frac{1 - \overline{b(\lambda)}b(z)}{1 - \overline{\lambda}z},\qquad \lambda,z \in \D.
\]
The norm in $\mathcal H(b)$ is denoted by $\|\cdot\|_b$ and the scalar product by $\scal{\cdot}{\cdot}_b$. In this paper, we shall be interested in the cyclicity of the shift operator on $\mathcal H(b)$. Let $S$ be the shift operator on $H^2$ defined as $(Sf)(z)=zf(z)$, $z\in\D$, $f\in H^2$. It is well-known that $\mathcal H(b)$ is invariant with respect to $S$ if and only if $\log(1-|b|)\in L^1(\T)$ \cite[Corollary 20.20]{DBR2}.

So {\bf from now on}, we will assume that $b$ is a non-constant function in the closed unit ball of $H^\infty$ which satisfies $\log(1-|b|)\in L^1(\T)$. In this case, there is a unique outer function $a \in H^\infty$ such that $a(0) > 0$ and $\abs{a}^2 + \abs{b}^2 = 1$ almost everywhere on $\T$. The function $a$ is called the \textit{Pythagorean mate} of $b$ and we say that $(a,b)$ forms a \textit{Pythagorean pair}. Let us recall a useful characterization of the membership to $\mathcal H(b)$ when $\log(1-|b|)\in L^1(\T)$: let $f\in H^2$, then 
\begin{equation}\label{eq:carac-membership-Hb}
f\in\mathcal H(b)\Longleftrightarrow \exists f^+\in H^2,\, T_{\bar b}f=T_{\bar a}f^+.
\end{equation}
Moreover, in this case, the function $f^+$ is unique (since $a$ is outer, then $T_{\bar a}$ is one-to-one) and 
\begin{equation}\label{caraHbNExt}
\|f\|_b^2=\|f\|_2^2+\|f^+\|_2^2.
\end{equation}
See \cite[Corollary 25.10]{DBR2}. If $f\in\mathcal H(b)$, according to \eqref{eq:carac-membership-Hb}, there exists a (unique) function $g\in H^2$ such that 
\begin{equation}\label{eq223:carac-membership-Hb}
\bar{b}f=\bar{a}f^+ +\overline{zg},\qquad \mbox{a.e. on }\T.
\end{equation}
It is also known that for every $\lambda\in\D$, $k_\lambda$ and $bk_\lambda$ belong to $\mathcal H(b)$, and for $f\in\mathcal H(b)$, we have
\begin{equation}\label{DecompScalfNoyrep}
    \scal{f}{k_\lambda}_b = f(\lambda) + \frac{b(\lambda)}{a(\lambda)}f^+(\lambda) \quad\text{and} \quad \scal{f}{bk_\lambda}_b = \frac{f^+(\lambda)}{a(\lambda)}.
\end{equation}
We  also recall that the set of polynomials $\mathcal P$ is dense in $\mathcal H(b)$ \cite[Theorem 23.13]{DBR2}. Moreover, if $(\lambda_n)_{n\geq 1}$ is a sequence of points in $\mathbb D$, then $\{k_{\lambda_n}:n\geq 1\}$ is complete in $\mathcal H(b)$ if and only if $\sum_{n=1}^\infty(1-|\lambda_n|)=\infty$  \cite[Theorem 4.2]{Rim}.

An important tool in the study of a reproducing kernel Hilbert space is its associated multiplier algebra. We denote by 
\[
\mathfrak{M}(\mathcal{H}(b)) := \{ \varphi \in \mbox{Hol}(\D) : \varphi f \in \mathcal{H}(b),\,\forall f \in \mathcal{H}(b) \},
\]
the multiplier algebra of $\mathcal H(b)$. 
Using the closed graph theorem, it is easy to see that if $\varphi\in\mathfrak{M}(\mathcal{H}(b))$, then $M_\varphi$, the multiplication operator  by $\varphi$, is bounded on $\mathcal H(b)$. The algebra of multipliers is a Banach algebra when equipped with the norm $\norm{\varphi}_{\mathfrak{M}(\mathcal{H}(b))} = \norm{M_\varphi}_{\Lcont(\mathcal{H}(b))}$.  Here $\mathcal L(\mathcal H(b))$ is the space of all linear and bounded operators on $\mathcal H(b)$ endowed with its usual norm. Using standard arguments, we see that  $\mathcal H(b)\subset H^\infty\cap \mathcal H(b)$. In general this inclusion is strict. Moreover, when $\log(1-|b|)\in L^1(\mathbb T)$, it is known that there are plenty of multipliers. More precisely, we have $\mbox{Hol}(\overline{\mathbb D})\subset \mathfrak{M}(\mathcal H(b))$. See \cite[Theorem 24.6]{DBR2}. In particular, for every $\lambda\in\mathbb D$ and every polynomial $p$, the functions $k_\lambda$ and $p$ are both multipliers of $\mathcal H(b)$.

In the theory of de the Branges-Rovnyak spaces, an important tool is the concept of Alexandrov-Clark measures. Recall that given $b$ in the closed unit ball of $H^\infty$ and $\alpha\in\T$, there exists a unique finite positive Borel measure $\mu_\alpha$ on $\T$ such that 
\begin{equation}\label{eq:AC-measures}
\frac{1-|b(z)|^2}{|1-\overline{\alpha} b(z)|^2}=\int_\T \frac{1-|z|^2}{|z-\zeta|^2}\,d\mu_\alpha(\zeta),\qquad z\in\D.
\end{equation}
The collection $\{\mu_\alpha\}_{\alpha\in\T}$ is called the family of Alexandrov-Clark measures of $b$. When $\log(1-|b|)\in L^1(\T)$, denoting by 
\[
F_\alpha=\frac{a}{1-\overline{\alpha}b},\qquad \alpha\in \T,
\]
and using standard facts on Poisson integrals, we see that 
\begin{equation}\label{sdssdqsdsqdqs22391343ZM23}
|F_\alpha(\zeta)|^2=\frac{d\mu_\alpha^{(a)}}{dm}(\zeta),\qquad\mbox{for a.e. }\zeta\in\T,
\end{equation}
where $\mu_\alpha^{(a)}$ denotes the continuous part of the measure $\mu_\alpha$ and $d\mu_\alpha^{(a)}/dm$ denotes its Radon-Nikodym derivative with respect to Lebesgue measure $m$. In particular, $F_\alpha\in H^2$ for every $\alpha\in\T$. Now, for $\alpha\in\T$, define by 
\[
T_{\overline{F_\alpha}}k_\lambda=\overline{F_\alpha(\lambda)}k_\lambda,\qquad \lambda\in\D.
\]
Then $T_{\overline{F_\alpha}}$ is a densely defined operator on $H^2$. See \cite[Section 13.6]{DBR1} for more information on Toeplitz operators with symbols in $L^2(\T)$. We just mention here a useful result on the properties of $T_{1-\overline{\alpha}b}T_{\overline{F_\alpha}}$.
\begin{lem}\label{Lem-Tunbounded-Clark}
Let $\alpha\in\T$. Then
\begin{enumerate}
\item[(i)] $T_{1-\overline{\alpha}b}T_{\overline{F_\alpha}}$ is an isometry from $H^2$ into $\mathcal H(b)$, and we have
\[
T_{1-\overline{\alpha}b}T_{\overline{F_\alpha}}H^2=\Span_{\mathcal H(b)}((1-\overline{\alpha}b)z^n:n\geq 0).
\]
\item[(ii)] $T_{1-\overline{\alpha}b}T_{\overline{F_\alpha}}$ is onto if and only if $\mu_\alpha$ is absolutely continuous with respect to $m$.
\end{enumerate}
\end{lem}
\begin{proof}
See \cite[Theorem 24.23]{DBR2} and \cite[Theorem 29.16]{DBR2}.
\end{proof}

An important subspace of $\mathcal H(b)$ is $\mathcal M(a)=aH^2$ equipped with the range norm
\[
\|af\|_{\mathcal M(a)}=\|f\|_2,\qquad f\in H^2.
\]
It is known \cite[Theorem 23.2]{DBR2} that $\mathcal M(a)$ is contractively contained into $\mathcal H(b)$, which means that for every $f\in H^2$, $af\in\mathcal H(b)$ and 
\begin{equation}\label{eq:Macontractivelycontained}
\|af\|_b\leq \|af\|_{\mathcal M(a)}=\|f\|_2.
\end{equation}
 Moreover, we can characterize the density of $aH^2$ in $\mathcal H(b)$ in terms of rigid functions. Recall that a function $f\in H^1$, $f\not\equiv 0$, is said to be {\emph{rigid}} if, for any function $g\in H^1$, $g\not\equiv 0$, the assumption
\[
\arg(g)=\arg(f),\qquad \mbox{a.e. on }\T,
\]
implies the existence of a positive real constant $\lambda$ such that $g=\lambda f$. It turns out that rigid functions in $H^1$, of unit norm, coincide with exposed points of the closed unit ball of $H^1$  \cite[Theorem 6.15]{DBR2}. Note that if $b(0)=0$ and if the Alexandrov-Clark measure $\mu_\alpha$ of $b$ is absolutely continuous with respect to $m$ for some $\alpha\in\T$, then using \eqref{eq:AC-measures} and \eqref{sdssdqsdsqdqs22391343ZM23}, we see that the function $F_\alpha=a/(1-\overline{\alpha}b)$ has a unit norm in $H^1$. Recall also that for almost every $\alpha\in\T$, $\mu_\alpha$ is absolutely continuous with respect to $m$  \cite[Theorem 24.19]{DBR2}.
\begin{theo}\label{thm-densite-ma-rigid}
Assume that $b(0)=0$ and let $\alpha\in\T$ such that $\mu_\alpha$ is absolutely continuous with respect to $m$. The following assertions are equivalent:
\begin{enumerate}
\item[(i)] the subspace $aH^2$ is dense in $\mathcal H(b)$;
\item[(ii)] the function $F_\alpha^2$ is rigid.
\end{enumerate}
\end{theo}
\begin{proof}
See \cite[Corollary 29.4]{DBR2}.
\end{proof}

As already said, when $\log(1-|b|)\in L^1(\mathbb T)$, the de Branges--Rovnyak space $\mathcal H(b)$ is invariant with respect to $S$. We shall denote by $S_b$ the restriction of $S$ to $\mathcal H(b)$, and it is known \cite[Theorem 24.3]{DBR2} that 
\begin{equation}\label{norm-Sb}
\|S_b\|=\frac{\sqrt{1-|b(0)|^2}}{|a(0)|}>1.
\end{equation}
 We shall be interested in the cyclic vectors of $S_b$. Recall that a function $f\in\mathcal H(b)$ is said to be \emph{cyclic} for $S_b$ if
 \[
\Span_{\mathcal H(b)}(S_b^n f: n \ge 0)=\Span_{\mathcal H(b)}(z^n f : n \ge 0)=\mathcal{H}(b).
\] 
This is of course equivalent to say that the set $\{pf:\,p\in\mathcal P\}$ is dense in $\mathcal H(b)$. By density of polynomials in $\mathcal H(b)$, the constant function $1$ is clearly cyclic for $S_b$. Moreover, standard arguments show that a function $f\in\mathcal H(b)$ is cyclic for $S_b$ if and only if there exists a sequence of polynomials $(p_n)_n$ satisfying $\|1-p_nf\|_b\to 0$ as $n\to\infty$. In other words, it is sufficient to approximate the function $1$ to get cyclicity. See \cite[Lemma 3.1]{CyclFG}. We also recall the following known fact which follows easily from Beurling's result and the fact that $\mathcal H(b)$ is contractively contained in $H^2$.
\begin{lem}\label{outer-is-necessary}
Let $f\in\mathcal H(b)$. Assume that $f$ is cyclic for $S_b$. Then $f$ is outer.
\end{lem}
\begin{proof}
See \cite[Lemma 3.2]{CyclFG}.
\end{proof}
We end these preliminaries on $\mathcal H(b)$ spaces by a computation which will be a key result for the completeness problem related to the cyclic vectors for $S_b$. 
\begin{lem}\label{lemme-fkplus}
   Let $f \in \mathcal{H}(b)$ and $\lambda \in \D$. Then $fk_\lambda\in\mathcal H(b)$ and
    \begin{equation}\label{decompfnoy+}
        (fk_\lambda)^+ = f^+ k_\lambda +  \frac{\overline{\lambda g(\lambda)}}{\overline{a(\lambda)}}k_\lambda,
    \end{equation}
    where $g \in H^2$ is defined by \eqref{eq223:carac-membership-Hb}.
\end{lem}

\begin{proof}
First recall that $k_\lambda\in\mathfrak{M}(\mathcal H(b))$ and thus $fk_\lambda\in\mathcal H(b)$. Now, using \eqref{eq223:carac-membership-Hb}, we have 
 \[
 T_{\overline{b}} (fk_\lambda) = P_+(\overline{b}f k_\lambda) = P_+(\overline{a}f^+ k_\lambda + \overline{z}\overline{g}k_\lambda) = P_+(\overline{a}f^+ k_\lambda) + P_+(\overline{z}\overline{g} k_\lambda).
 \] 
Observe that $P_+(\overline{z} \overline{g}k_\lambda) = \overline{\lambda g(\lambda)}k_\lambda$. Indeed, for all $h \in H^\infty$, we have 
\[
\scal{h}{P_+(\overline{z}\overline{g}k_\lambda)}_2 = \scal{h}{\overline{z}\overline{g}k_\lambda}_2 = \scal{z g h}{k_\lambda}_2 = \lambda g(\lambda) h(\lambda) = \scal{h}{\overline{\lambda g(\lambda)}k_\lambda}_2,
\] 
and we conclude thanks to the density of $H^\infty$ in $H^2$. Using now the fact that $k_\lambda = T_{\overline{a}} \left (\frac{k_\lambda}{\overline{a(\lambda)}} \right)$, we get 
    \begin{equation*}
        T_{\overline{b}} (fk_\lambda) = T_{\overline{a}}(f^+ k_\lambda) + T_{\overline{a}} \left (\frac{\overline{\lambda g(\lambda)}}{\overline{a(\lambda)}}k_\lambda \right) = T_{\overline{a}} \left (f^+ k_\lambda + \frac{\overline{\lambda g(\lambda)}}{\overline{a(\lambda)}}k_\lambda \right). 
    \end{equation*}
 The definition of $(fk_\lambda)^+$ completes the proof. 
\end{proof}

\subsection{Cauchy transform and distribution function} 
A useful tool in the proof of our main result on cyclic vectors of $S_b$ will be the Cauchy transform. Recall that the Cauchy transform of a function $f \in L^1(\T)$ is defined by 
\[
C f(z) = \int_\T \frac{f(\zeta)}{1 - \overline{\zeta}z} ~dm(\zeta),\qquad z \in \C \backslash \T.
\]
Note that $Cf \in \mbox{Hol}(\C \backslash \T)$ and in particular, for $z \in \D$, we have
\[
Cf(z) = \sum_{n \ge 0} \widehat{f}(n) z^n.
\]
We refer the reader to \cite{CauchyTransf, DBR1} for more details about Cauchy transforms. We shall here just recall the main useful properties for us. It is not difficult to prove that for every $f\in L^1(\mathbb T)$, $Cf\in \bigcap_{0 < p < 1} H^p$, and so, in particular, $Cf$ belongs to the Smirnov class $\mathcal N^+$. In general, $Cf$ does not belong to $L^1(\mathbb T)$ but we have a weaker version due to A.N. Kolmogorov. For that purpose, let us recall that $L_0^{1, \infty}(\T)$ is the space of measurable function $h:\T\to\C$ satisfying 
\[
\lambda_h(t) = o\left(\frac{1}{t}\right),\qquad \mbox{as }t \to \infty,
\]
where $\lambda_h(t)=m\left ( \left \{ \zeta \in \T : \abs{h(\zeta)} > t \right \} \right)$, $t>0$, is the distribution function of $h$. It is easy to see that $L^1(\T)\subset L_0^{1, \infty}(\T)$, and if $h\in L_0^{1, \infty}(\T)$ and $\varphi\in L^\infty(\T)$, then $\varphi h\in L_0^{1, \infty}(\T)$.
\begin{theo}[Kolmogorov]\label{Kolmogorov} 
    Let $f \in L^1(\T)$. Then $Cf \in L_0^{1, \infty}(\T)$. 
\end{theo}
\begin{proof}
See \cite[Proposition 3.4.11]{CauchyTransf}.
\end{proof}
In particular, the Cauchy transform $C$ maps $L^1(\T)$ into the space $H_0^{1,\infty}=L_0^{1, \infty}(\T)\cap\mathcal N^+$.

Two representations of functions in $H_0^{1,\infty}$ will be useful for us.
\begin{theo}[Aleksandrov]\label{Alexandrov}
Let $f \in H_0^{1, \infty}$. Then, for all $z \in \D$, we have
\begin{equation}\label{eq:Alexandrov1}
f(z) = \lim_{A \rightarrow + \infty} \int_{\abs{f} \le A} \frac{f(\zeta)}{1 - \overline{\zeta}z}\,dm(\zeta),
\end{equation}
and
\begin{equation}\label{eq:Alexandrov2}
f(0) = \lim_{A \rightarrow + \infty} \int_{\abs{f} \le A} \frac{f(\zeta)}{1 - \overline{z}\zeta}\,dm(\zeta).
\end{equation}
\end{theo}
\begin{proof}
See \cite[Theorem 6]{Aleksandrov} and \cite[Lemma 5.13 and Lemma 5.22]{CauchyTransf}.
\end{proof}
We shall use the following simple consequence.
\begin{coro}\label{FctFctConjHcst}
Let $f, \overline{f} \in H_0^{1, \infty}$. Then $f$ is a constant function.
\end{coro}
\begin{proof}
According to \eqref{eq:Alexandrov1}, since $\overline{f}\in H_0^{1, \infty}$, we have
\[
\overline{f(z)} = \lim_{A \rightarrow + \infty} \int_{\abs{\overline{f}} \le A} \frac{\overline{f(\zeta)}}{1 - \overline{\zeta}z}\,dm(\zeta) =\overline{ \lim_{A \rightarrow + \infty} \int_{\abs{f} \le A} \frac{f(\zeta)}{1 - \zeta \overline{z}}\,dm(\zeta)},
\]
for all $z\in\D$. Now, since $f\in H_0^{1, \infty}$, \eqref{eq:Alexandrov2} gives
\[
\lim_{A \rightarrow + \infty} \int_{\abs{f} \le A} \frac{f(\zeta)}{1 - \zeta \overline{z}} ~dm(\zeta) = f(0),
\]
for all $z\in\D$. Finally, we get that  $\overline{f(z)} = \overline{f(0)}$ for all $z \in \D$ and $f$ is a constant function.
\end{proof}

Using F. and M. Riesz Theorem, it is easy to check that for $f\in L^1(\T)$, then $Cf\equiv 0$ on $\D$ if and only if $f\in \overline{H_0^1}=\overline{zH^1}$. We end this section on Cauchy transform by two simple identities. 
\begin{lem}\label{Lem-Cauchy-Transform}\

\begin{enumerate}
\item[(i)] Let $f \in L^1(\T)$. Then $C(f) - f = \widehat{f}(0) - \overline{C(\overline{f})}$ a.e. on $\T$.
\item[(ii)] For $f,g \in H^2$, $\scal{f}{g k_\lambda}_2 = C(f\overline{g})(\lambda)$ for all $\lambda \in \D$. 
\end{enumerate}
\end{lem}
\begin{proof}
$(i)$ Let $\zeta \in \T$ and $z \in \D$. Observe that  
\[
- 1 + \frac{\zeta}{\zeta - z} + \frac{\overline{\zeta}}{\overline{\zeta} - \overline{z}} = \frac{1 - \abs{z}^2}{\abs{\zeta - z}^2},
\] 
and then
\[ 
- \int_\T f(\zeta) ~dm(\zeta) + \int_\T f(\zeta) \frac{\zeta}{\zeta - z} ~dm(\zeta) + \int_\T f(\zeta)\frac{\overline{\zeta}}{\overline{\zeta} - \overline{z}} ~dm(\zeta) = (Pf)(z),
\]
where $Pf$ is the Poisson integral of $f$. In other words, we have 
\[
- \widehat{f}(0) + \int_\T \frac{f(\zeta)}{1 - \overline{\zeta}z} ~dm(\zeta) + \int_\T \frac{f(\zeta)}{1 - \overline{z}\zeta} ~dm(\zeta) = (Pf)(z)
\] 
and thus
\[
- \widehat{f}(0) + (Cf)(z) + \overline{C(\overline{f})(z)} = (Pf)(z),
\] 
for all $z\in\D$.  Since $C$ maps $L^1(\T)$ into $\mathcal N^+$, $Cf$ and $C(\bar f)$ have radial limits almost everywhere on $\T$. Moreover, $Pf$ tends radially to $f$ almost everywhere on $\T$. Thus, we get that
\[
- \widehat{f}(0) + Cf + \overline{C(\overline{f})} = f \qquad\mbox{a.e. on }\T,
\]
which proves $(i)$.

$(ii)$ Let $f,g \in H^2$ and $\lambda \in \D$. We have 
\begin{equation*}
\scal{f}{gk_\lambda}_2 = \int_\T f(\zeta) \frac{\overline{g(\zeta)}}{1 - \lambda \overline{\zeta}} ~ dm(\zeta) = C(f\overline{g})(\lambda). \qedhere 
\end{equation*}
\end{proof}

\section{A completeness problem related to cyclicity}\label{Firststeps}
We shall make an interesting connection between the cyclic vectors for $S_b$ and a completeness problem. This link will enable us to use Lemma~\ref{lemme-fkplus} to give some sufficient conditions for the cyclicity in the next section. We first start with the situation in the Hardy space to motivate our results on $\mathcal H(b)$. 

\subsection{A completeness problem in the Hardy space}

\begin{lem}\label{Linfnoy2-H2H2}
Let $f=\Theta f_e\in H^2$, where $\Theta$ is inner and $f_e$ is outer. Then, for every sequence $(\lambda_n)_{n\geq 1}\subset \D$ satisfying $\sum_{n=1}^\infty(1-|\lambda_n|)=\infty$, we have
\[
\Span_{H^2}(fk_{\lambda_n}:n\geq 1)=\Theta H^2.
\]
\end{lem}
\begin{proof}
Since $fk_{\lambda_n}\in\Theta H^2$ and $\Theta H^2$ is a closed subspace of $H^2$, the inclusion $\Span_{H^2}(fk_{\lambda_n}:n\geq 1)\subset\Theta H^2$ is clear. Assume now, on the contrary, that $\Span_{H^2}(fk_{\lambda_n}:n\geq 1)\subsetneq\Theta H^2$. Then there exists $\varphi \in H^2$, $\varphi\neq 0$, such that $\Theta \varphi$ is orthogonal to   $fk_{\lambda_n}$ in $H^2$ for every $n \ge 1$. According to Lemma~\ref{Lem-Cauchy-Transform} $(ii)$, we have
\[
0=\scal{\Theta \varphi}{f k_{\lambda_n}}_2=\scal{\varphi}{f_e k_{\lambda_n}}_2=C(\varphi \overline{f_e})(\lambda_n).
\] 
But, since $C(\varphi\overline{f_e})\in \mathcal N^+$ and $\sum_{n=1}^{\infty} (1 - \abs{\lambda_n}) =\infty$, we deduce that $C(\varphi \overline{f_e})\equiv 0$ on $\D$. In particular, there exists $\psi\in H_0^1$ such that $\varphi\overline{f_e}=\overline{\psi}$, which gives that $\varphi=\frac{\overline{\psi}}{\overline{f_e}}$ almost everywhere on $\T$. In particular, we have $\overline{\varphi}\in L^2(\T)\cap\mathcal N^+=H^2$, and then $\varphi\in H^2\cap \overline{H_0^2}=\{0\}$, which gives the desired contradiction.
\end{proof}

\begin{coro}\label{cor:completeness-Hardy}
Let $(\lambda_n)_{n\geq 1}\subset \D$ satisfying $\sum_{n=1}^\infty(1-|\lambda_n|)=\infty$, and let $f\in H^2$. Then 
\[
\Span_{H^2}(fk_{\lambda_n}:n\geq 1)=H^2\Longleftrightarrow\mbox{ $f$ is outer.}
\]
\end{coro}
\begin{proof}
This follows immediately from Lemma~\ref{Linfnoy2-H2H2}.
\end{proof}
\begin{rem}
It should be noted that if the sequence $(\lambda_n)_{n\geq 1}$ is a Blaschke sequence, that is $\sum_{n=1}^\infty (1-|\lambda_n|)<\infty$, then it may happen that the sequence $(fk_{\lambda_n})_{n\geq 1}$ is no longer complete, even if $f$ is an outer function. Indeed, it is sufficient to observe that if $B$ is the  Blaschke product, associated to the Blaschke sequence $(\lambda_n)_{n\geq 1}$, then
\[
\Span_{H^2}(k_{\lambda_n}:n\geq 1)=H^2\ominus BH^2\neq H^2.
\]
\end{rem}
\subsection{A completeness problem in \texorpdfstring{$\mathcal H(b)$}{}} 

The cyclicity of $S_b$ involves the closed linear span generated by $z^nf$, $n\geq 1$, where $f\in\mathcal H(b)$. We shall make a connection with the closed linear span generated by $k_{\lambda}f$, $|\lambda|<\|S_b\|^{-1}$. Recall that $\|S_b\|>1$, see \eqref{norm-Sb}. The following lemma will be the key  to make this connection. It is inspired by \cite[Theorem 5.5]{DBR1}.
\begin{lem}\label{egSpanMult}
We have
\begin{equation}\label{eq:Span-Mult2434Z}
\Span_{\mathfrak{M}(\mathcal{H}(b))}(z^n : n \ge 0)= \Span_{\mathfrak{M}(\mathcal{H}(b))}\left(k_\mu : \abs{\mu} < \frac{1}{\norm{S_b}}\right).
\end{equation}
\end{lem}

\begin{proof}
We shall prove the two inclusions separately. First, let $\mu\in\C$ such that $|\mu|<{\norm{S_b}}^{-1}$. Let us prove that $k_\mu\in \Span_{\mathfrak{M}(\mathcal{H}(b))}(z^n : n \ge 0)$. Observe that 
\begin{align*}
\norm{k_\mu - \sum_{n=0}^N \overline{\mu}^n z^n}_{\mathfrak{M}(\mathcal{H}(b))} & \le \sum_{n \ge N+1} \abs{\mu}^n \norm{z^n}_{\mathfrak{M}(\mathcal{H}(b))} = \sum_{n \ge N+1} \abs{\mu}^n \norm{S_b^n}\\ & \le \sum_{n \ge N+1} \abs{\mu}^n \norm{S_b}^n\underset{N \rightarrow \infty}{\longrightarrow} 0.
        \end{align*}
Thus we get that $k_\mu \in \Span_{\mathfrak{M}(\mathcal{H}(b))}(z^n : n \ge 0)$, which gives the first inclusion.

For the reversed inclusion, let $n\in\mathbb N$ and let us prove that $z^n$ belongs to  $\Span_{\mathfrak{M}(\mathcal{H}(b))}\left(k_\mu : \abs{\mu}<{\norm{S_b}}^{-1}\right)$. 
 For $n=0$, this is trivial because $z^0=k_0=1$. Let us assume now that $n\geq 1$, and introduce the $n$th root of unity $\zeta=e^{2i\pi/n}$. For $0<r<\|S_b\|^{-1}$, define 
\[
g_r=\frac{k_r + k_{r\zeta} + \cdots + k_{r\zeta^{n-1}} - nk_0}{n r^n}
\]
which belongs to $\bigvee(k_\mu:|\mu|<\|S_b\|^{-1})$. We shall prove that $g_r \underset{r \rightarrow 0}{\longrightarrow} z^n$ in $\mathfrak{M}(\mathcal{H}(b))$. Indeed, for all $z\in\D$, we have 
    \begin{align*}
g_r(z) & = \frac{1}{n r^n} \left (\sum_{k=0}^\infty r^k z^k + \sum_{k=0}^\infty r^k \zeta^k z^k + \cdots + \sum_{k=0}^\infty r^k \zeta^{(n-1)k}z^k - n \right) \\ & = \frac{1}{n r^n} \sum_{k=1}^\infty r^k z^k \left (\sum_{i = 0}^{n-1} \zeta^{ik} \right).
    \end{align*}
Note that, for $k\in\mathbb N^*$, we have
\[
\sum_{i=0}^{n-1}\zeta^{ik}=\begin{cases}
n,&\mbox{if }k=n\ell\mbox{ for some }\ell\in\mathbb N^*, \\
\frac{1-\zeta^{nk}}{1-\zeta^k}=0,& \mbox{if }k\notin n\mathbb N.
\end{cases}
\]
Thus we get that     
\[
g_r(z) = \frac{1}{r^n} \sum_{\ell=1}^\infty r^{\ell n}z^{\ell n} = \sum_{\ell=1}^\infty r^{(\ell-1)n}z^{\ell n}.
\] 
This implies
\begin{align*}
\norm{g_r- z^n}_{\mathfrak{M}(\mathcal{H}(b))} & \le \sum_{\ell=2}^\infty r^{(\ell-1)n} \norm{S_b}^{\ell n}= \sum_{\ell=1}^\infty r^{\ell n}\norm{S_b}^{(\ell+1)n}\\ & 
=\frac{r^n \|S_b\|^{2n}}{1-r^n\|S_b\|^n}, 
\end{align*}
 and the last term tends to $0$ as $r$ goes to $0$. Hence  $z^n$ belongs to $\Span_{\mathfrak{M}(\mathcal{H}(b))}\left(k_\mu : \abs{\mu}<{\norm{S_b}}^{-1}\right)$, and we get the reversed inclusion, and then the equality \eqref{eq:Span-Mult2434Z}.
 \end{proof}

\begin{coro}\label{nvlExpCyclShift}
Let $f \in \mathcal{H}(b)$. Then 
\[
\Span_{\mathcal H(b)}(z^n f : n \ge 0) = \Span_{\mathcal H(b)}\left (k_\mu f : \abs{\mu} < \frac{1}{\norm{S_b}} \right).
\]
In particular, $f$ is cyclic in $\mathcal H(b)$ if and only if 
\[
 \Span_{\mathcal H(b)}\left (k_\mu f : \abs{\mu} < \frac{1}{\norm{S_b}} \right)=\mathcal H(b).
\]
\end{coro}
\begin{proof}
This follows directly from Lemma~\ref{egSpanMult}.
\end{proof}

\begin{lem}\label{SpanLatSb}
Let $(\lambda_n)_{n \ge 1} \subset \D \backslash \{0\}$ satisfying $\sum_{n=1}^\infty (1 - \abs{\lambda_n})=\infty$, and let $f \in \mathcal{H}(b)$. Then $\Span_{\mathcal H(b)}(fk_{\lambda_n} : n \ge 1)$ is invariant with respect to $S_b$.
\end{lem}

\begin{proof}
Observe that, since $\lambda_n\neq 0$, we have
\begin{align*}
            S_b(fk_{\lambda_n})&=zf \frac{1}{1 - \overline{\lambda_n}z}= -\frac{1}{\overline{\lambda_n}} f \left (\frac{1 - \overline{\lambda_n}z - 1}{1 - \overline{\lambda_n}z} \right) \\ 
& = -\frac{1}{\overline{\lambda_n}}(f-fk_{\lambda_n}).
\end{align*}    
Thus it is sufficient to prove that $f\in\Span_{\mathcal H(b)}(fk_{\lambda_n} : n \ge 1)$. Assume on the contrary that $f\notin\Span_{\mathcal H(b)}(fk_{\lambda_n} : n \ge 1)$. Then there exists $h \in \mathcal{H}(b)$ such that $\scal{h}{fk_{\lambda_n}}_b=0$ for all $n \ge 1$, and $\scal{h}{f}_b \neq 0$. According to \eqref{caraHbNExt} and \eqref{decompfnoy+}, we get
\begin{align*}
        0 & = \scal{h}{fk_{\lambda_n}}_2 + \scal{h^+}{(fk_{\lambda_n})^+}_2 \\ & = \scal{h}{fk_{\lambda_n}}_2 + \scal{h^+}{f^+k_{\lambda_n}}_2 + \frac{\lambda_n g(\lambda_n)}{a(\lambda_n)}h^+(\lambda_n),
\end{align*}       
where $g$ satisfies \eqref{eq223:carac-membership-Hb}. Now it follows from Lemma~\ref{Lem-Cauchy-Transform} $(ii)$ that 
\[
0= C(h\overline{f})(\lambda_n) + C(h^+ \overline{f^+})(\lambda_n) + \frac{\lambda_n g(\lambda_n)}{a(\lambda_n)}h^+(\lambda_n). 
\]
Multiplying by $a(\lambda_n)$ gives 
\[
0 = a(\lambda_n)C(h\overline{f})(\lambda_n) + a(\lambda_n)C(h^+ \overline{f^+})(\lambda_n) + \lambda_n g(\lambda_n)h^+(\lambda_n).
\]
Observe that the function $aC(h\overline{f}) + aC(h^+\overline{f^+}) + z g h^+$ belongs to $\mathcal N^+$ and vanishes on the sequence $(\lambda_n)_{n\geq 1}$. Since $\sum_{n=1}^\infty (1 - \abs{\lambda_n})=\infty$, we deduce that $aC(h\overline{f}) + aC(h^+\overline{f}^+) + z g h^+ \equiv 0$ on $\D$. In particular, 
\[
 a(0)C(h \overline{f})(0) + a(0)C(h^+\overline{f^+})(0) = 0.
 \] 
Using  $a(0) \neq 0$ and one more time Lemma~\ref{Lem-Cauchy-Transform} $(ii)$  and \eqref{caraHbNExt}, we deduce that
\[
0 = C(h \overline{f})(0) + C(h^+\overline{f^+})(0) = \scal{h}{f}_2 + \scal{h^+}{f^+}_2 = \scal{h}{f}_b.
\]
Thus we get a contradiction, and we can conclude that  $f$ belongs to $\Span_{\mathcal H(b)}(fk_{\lambda_n} : n \ge 1)$.
 \end{proof}
 
 We may now link the problem of cyclicity with our problem of completeness. 
\begin{coro}\label{CaraCyclSbSpanHb}
Let $f \in \mathcal{H}(b)$. The following assertions are equivalent:
\begin{enumerate}
\item[(i)] $f$ is cyclic for $S_b$; 
\item[(ii)] for every sequence $(\lambda_n)_{n \ge 1} \subset \D$ satisfying $\sum_{n=1}^\infty (1 - \abs{\lambda_n}) = \infty$, we have $\Span_{\mathcal H(b)}(fk_{\lambda_n} : n \ge 1) = \mathcal{H}(b)$.
\end{enumerate}
\end{coro}

\begin{proof}
$(i)\implies (ii)$: Assume that $f\in\mathcal H(b)$ is a cyclic vector for $S_b$ and let $(\lambda_n)_{n \ge 1} \subset \D$ satisfying $\sum_{n=1}^\infty (1 - \abs{\lambda_n}) = \infty$. Without loss of generality, we may suppose that for all $n \ge 1$, $\lambda_n \neq 0$. According to Lemma \ref{SpanLatSb}, $S_b \Span_{\mathcal H(b)}(fk_{\lambda_n} : n \ge 1) \subset \Span_{\mathcal H(b)}(fk_{\lambda_n} : n \ge 1)$. In particular, for all polynomials $p$, we have $pf\in\Span_{\mathcal H(b)}(fk_{\lambda_n} : n \ge 1)$, which gives that
\[
\Span_{\mathcal H(b)}(pf:p\in\mathcal P)\subset \Span_{\mathcal H(b)}(fk_{\lambda_n} : n \ge 1).
\]
We conclude using the cyclicity of $f$.

$(ii)\implies (i)$: Let $(\lambda_n)_{n\geq 1}$ be any sequence such that $|\lambda_n|<\norm{S_b}^{-1}$. Since $\norm{S_b}>1$, we obviously get that $\sum_{n=1}^\infty(1-|\lambda_n|)=\infty$, whence, by hypothesis, we have $\Span_{\mathcal H(b)}(fk_{\lambda_n}:n\geq 1)=\mathcal H(b)$. It remains now to apply Corollary~\ref{nvlExpCyclShift} to conclude that $f$ is cyclic for $S_b$. 
     \end{proof}

\section{Main result on the completeness problem}\label{MainResult}

In this subsection, we shall discuss the problem of completeness of the sequence $(fk_{\lambda_n})_{n\geq 1}$ in $\mathcal H(b)$, where $f\in\mathcal H(b)$ and 
$(\lambda_n)_{n\geq 1}\subset \D$ satisfies $\sum_{n=1}^\infty(1-|\lambda_n|)=\infty$. We start with a simple observation. 
\begin{lem}
Let $(\lambda_n)_{n\geq 1}\subset\D$ such that $\sum_{n=1}^\infty(1-|\lambda_n|)=\infty$, and let $f\in\mathcal H(b)$. If 
\[
\Span_{\mathcal H(b)}(fk_{\lambda_n}:n\geq 1)=\mathcal H(b),
\]
then $f$ is an outer function.
\end{lem}
\begin{proof}
Let $p$ be a polynomial and $\varepsilon>0$. Then there exists a function $g\in\bigvee(fk_{\lambda_n}:n\geq 1)$ such that $\|p-g\|_b\leq \varepsilon$. Using \eqref{caraHbNExt}, we see that $\|p-g\|_2\leq\varepsilon$, and then $p\in\Span_{H^2}(fk_{\lambda_n}:n\geq 1)$. But $\mathcal P$ is dense in $H^2$, whence 
\[
\Span_{H^2}(fk_{\lambda_n}:n\geq 1)=H^2.
\]
Now Corollary~\ref{cor:completeness-Hardy} implies that $f$ is outer.
\end{proof}
When one studies the problem of completeness of $(fk_{\lambda_n})_{n\geq 1}$ in $\mathcal H(b)$, we may therefore assume, without loss of generality, that $f$ is outer. 

\begin{theo}\label{MHypfcyclSb}
    Let $f \in \mathcal{H}(b)$ be an outer function and let $(\lambda_n)_{n \ge 1} \subset \D$ satisfying $\sum_{n=1}^\infty (1 - \abs{\lambda_n}) =\infty$. If $\frac{b}{f} \in L^\infty(\mathbb T)$, then 
\begin{equation}\label{eq:completude-MHypfcyclSb}
\Span_{\mathcal H(b)}(fk_{\lambda_n} : n \ge 1) = \mathcal{H}(b).
\end{equation}
In particular, $f$ is cyclic for $S_b$.
\end{theo}

\begin{proof}
Let $h \in \mathcal{H}(b)$ such that $h \perp_b fk_{\lambda_n}$ for all $n \ge 1$. Using the same computations as in the proof of Lemma \ref{SpanLatSb}, we get 
\begin{equation}\label{Eq1MHypfcyclSb}
aC(h\overline{f}) + aC(h^+ \overline{f^+}) + zgh^+ \equiv 0,\qquad \mbox{on $\D$,}
\end{equation}
where $g$ satisfies $\overline{b}f = \overline{a}f^+ + \overline{z}\overline{g}$ a.e. on $\T$. Taking the radial limits in \eqref{Eq1MHypfcyclSb} and using the definition of $g$, we deduce that
\[
aC(h\overline{f})+aC(h^+\overline{f^+})-a\overline{f^+}h^+ +b\overline{f}h^+=0\qquad \mbox{a.e. on }\T.
\]
Multiplying by $\overline{b}$ and using the relation $|a|^2+|b|^2=1$ a.e. on $\T$, we obtain
\[
a\left(\overline{b}C(h\overline{f})+\overline{b}C(h^+\overline{f^+})-\overline{bf^+}h^+ -\overline{af}h^+ \right)+\overline{f}h^+=0\qquad \mbox{a.e. on }\T.
\]
Since $a\neq 0$ and $f\neq 0$ almost everywhere on $\T$, we get
\[
\frac{1}{\overline{f}}\left(\overline{b}C(h\overline{f})+\overline{b}C(h^+\overline{f^+})-\overline{bf^+}h^+ -\overline{af}h^+\right)=-\frac{h^+}{a}.
\]
Using the fact that $h\in \mathcal H(b)$, there exists $h_1\in H^2$ satisfying $\overline{b}h=\overline{a}h^++\overline{zh_1}$ a.e. on $\T$, which gives
\[
\frac{1}{\overline{f}}\left(\overline{b}C(h\overline{f}+h^+\overline{f^+})-\overline{bf^+}h^+ -\overline{bf}h+\overline{zh_1f}\right)=-\frac{h^+}{a}.
\]
If $\varphi=h\overline{f}+h^+\overline{f^+}$, we can rewrite this identity as
\begin{equation}\label{eq2343434mjljlzjmlsqj}
\frac{h^+}{a}=-\frac{\overline{b}(C(\varphi)-\varphi)+\overline{zh_1f}}{\overline{f}}\qquad \mbox{a.e. on }\T.
\end{equation}
We claim that $\frac{h^+}{a}\in H_0^{1,\infty}\cap\overline{H_0^{1,\infty}}$. Indeed, on one hand, observe that $\frac{h^+}{a}\in\mathcal N^+$, and since $\varphi\in L^1(\T)$, by Kolmogorov Theorem, $C(\varphi)-\varphi$ belongs to $L^{1,\infty}_0(\T)$. But using that $\frac{\overline{b}}{\overline{f}}\in L^\infty(\T)$, we get that 
\[
\frac{\overline{b}}{\overline{f}}(C(\varphi)-\varphi)\in L_0^{1,\infty}(\T).
\]
Moreover, $\overline{zh_1}\in L^2(\T)\subset L^1(\T)\subset L_0^{1,\infty}(\T)$, whence it follows from \eqref{eq2343434mjljlzjmlsqj} that $\frac{h^+}{a}\in L_0^{1,\infty}(\T)$ and thus $\frac{h^+}{a}\in H_0^{1,\infty}$. On the other hand, according to Lemma~\ref{Lem-Cauchy-Transform} $(i)$, we have
\begin{equation}\label{erefssd123946}
\frac{\overline{h^+}}{\overline{a}}=-\frac{b}{f}\left(\overline{C(\varphi)}-\overline{\varphi}\right)-zh_1=-\frac{b}{f}\left(\overline{\widehat{\varphi}(0)}-C(\overline{\varphi})\right)-zh_1.
\end{equation}
Using once more  Kolmogorov Theorem, this identity shows that $\frac{\overline{h^+}}{\overline{a}}\in H_0^{1,\infty}$. Finally, we have $\frac{h^+}{a}\in H_0^{1,\infty}\cap \overline{H_0^{1,\infty}}$. It follows from Corollary~\ref{FctFctConjHcst} that $\frac{h^+}{a}$ is a constant function, say equals to $c\in\C$. In particular, with \eqref{erefssd123946} and Theorem~\ref{eq:Alexandrov1}, we obtain that
\[
-\frac{b(z)}{f(z)}\left(\overline{\widehat{\varphi}(0)}-C(\overline{\varphi})(z)\right)-zh_1(z)=\overline{c}
,\qquad z\in\D.\]
Evaluating at $z=0$ gives
\[
-\frac{b(0)}{f(0)}\left(\overline{\widehat{\varphi}(0)}-C(\overline{\varphi})(0)\right)=\overline{c}.
\]
But 
\[
C(\overline{\varphi})(0)=\int_\T \overline{\varphi(\xi)}\,dm(\xi)=\overline{\widehat{\varphi}(0)},
\]
whence $c=0$. In other words, $h^+=0$ and coming back to \eqref{Eq1MHypfcyclSb}, we deduce that $C(h\overline{f})\equiv 0$ on $\D$. In particular, there exists $\psi\in H_0^1$ such that $h\overline{f}=\overline{\psi}$. Hence $h=\frac{\overline{\psi}}{\overline{f}}\in H^2\cap \overline{H_0^2}=\{0\}$, which concludes the proof.
\end{proof}

\begin{rem}
The sufficient condition $b/f\in L^\infty(\T)$ implies that the set 
\[
Z(f):=\{\zeta\in\T:\liminf_{\substack{z\to\zeta\\z\in\D}}|f(z)|=0\}
\]
is included in the boundary spectrum of $b$, where we recall that the boundary of $b$ is defined as 
\[
\sigma(b)=\{\zeta\in\T:\liminf_{\substack{z\to\zeta\\z\in\D}}|b(z)|<1\}.
\]
Indeed, let us assume that there exists $\zeta\in Z(f)$ and $\zeta\notin\sigma(b)$. Then there exists a sequence $(z_n)_{n\geq 1}$ in $\D$ satisfying $z_n\to \zeta$ and $|f(z_n)|\to 0$, $n\to \infty$, and since $\zeta\notin\sigma(b)$, $\lim_{z\to\zeta}|b(z)|=1$, which implies that $|b(z_n)|\to 1$. This contradicts the fact that $b/f\in H^\infty$. 
\end{rem}

\begin{rem}
It should be noted that $b/f\in L^\infty(\T)$ is not necessary to have \eqref{eq:completude-MHypfcyclSb}. Indeed, let $b(z)=(1+z)/2$ and $f(z)=z-i$, $z\in\D$. Then $f\in\mathcal H(b)$ and since $f(1)=1-i\neq 0$, Corollary 4.2 from \cite{CyclNonExtr} implies that $f$ is cyclic for $S_b$. In particular, according to Corollary~\ref{CaraCyclSbSpanHb}, the property \eqref{eq:completude-MHypfcyclSb} is satisfied. However, $b/f \notin L^\infty(\T)$. 
\end{rem}

\section{Some consequences on cyclicity}\label{Consequences}
In this section, we shall discuss some easy consequences of our main theorem on cyclicity. 
\begin{coro}\label{CarabCyclSbExt}
    \begin{enumerate}
        \item[(i)] Let $f\in \mathcal{H}(b)$ satisfying $\inf_\D \abs{f} > 0$. Then $f$ is cyclic for $S_b$.
        \item[(ii)] The function $b$ is cyclic for $S_b$ if and only if $b$ is outer.
    \end{enumerate}
\end{coro}

\begin{proof}
 $(i)$ First it is known that the condition $\inf_\D \abs{f} > 0$ implies that $f$ is outer. See \cite[Page 67]{Nikolski}. Moreover, $\frac{b}{f}\in L^\infty(\T)$ and Theorem~\ref{MHypfcyclSb} implies that $f$ is cyclic.

$(ii)$ The second assertion follows immediately from Lemma~\ref{outer-is-necessary} and Theorem~\ref{MHypfcyclSb}. 
 \end{proof}
 
 Note that Corollary~\ref{CarabCyclSbExt} $(i)$ generalizes \cite[Corollary 3.4]{CyclFG}, which was proved under the additional  assumption that $f\in\mbox{Hol}(\overline{\D})$.   
\begin{expl}
(a) Let $\lambda\in\D$. Then 
\[
|k_\lambda^b(z)|=\left|\frac{1-\overline{b(\lambda)}b(z)}{1-\overline{\lambda}z}\right|\geq \frac{1-|b(\lambda)|}{2}>0.
\]
Thus, it follows from Corollary~\ref{CarabCyclSbExt} that $k_\lambda^b$ is cyclic for $S_b$.

(b) If $b$ is outer and $\lambda\in\D$, then $bk_\lambda$ is outer (being the product of two outer functions) and $b/bk_\lambda=1/k_\lambda=1-\overline{\lambda}z\in L^\infty(\T)$. Thus Theorem~\ref{MHypfcyclSb} implies that $bk_\lambda$ is cyclic.

Note that these two results on $k_\lambda^b$ and $bk_\lambda$ were already observed in \cite[Proposition 5.7]{CyclFG} but with the additional assumption that $(a,b)$ is a corona pair, meaning that $\inf_\D(|a|+|b|)>0$. We see here that this  assumption can be omitted. 
\end{expl}

For the next result, recall the definition of the family of Alexandrov-Clark measures $\{\mu_\xi\}_{\xi\in\T}$ associated to $b$. See \eqref{eq:AC-measures}.
\begin{coro}\label{Des1mcbcyclSb}
Let $c\in\C$. The function $1-cb$ is cyclic for $S_b$ if and only if we are in one of the following three cases:
\begin{enumerate}
\item[(i)] $|c|<1$;
\item[(ii)] $|c|=1$ and $\mu_{\overline c}$ is absolutely continuous with respect to Lebesgue measure $m$;
\item[(iii)] $|c|>1$ and $1-cb$ is outer.
\end{enumerate}
 \end{coro}

\begin{proof}
$(i)$ Assume that $|c|<1$. Observe that 
\[
\inf_\D \abs{1 - cb} \ge 1 - \abs{c} > 0,
\] 
and Corollary \ref{CarabCyclSbExt} implies that $1-cb$ is cyclic for $S_b$. 

$(ii)$ Assume now that $|c|=1$ and denote by $\xi=\overline{c}$. Using Lemma \ref{Lem-Tunbounded-Clark} $(i)$, we have 
\[
\Span_{\mathcal H(b)}((1-cb)z^n:n\geq 0)=T_{1-\overline{\xi}b}T_{\overline{F_\xi}}H^2,
\]
where $F_\xi$ is the $H^2$ outer function defined by $F_\xi=\frac{a}{1-\overline{\xi}b}$. Hence, $1-cb$ is cyclic for $S_b$ if and only if $T_{1-\overline{\xi}b}T_{\overline{F_\xi}}H^2=\mathcal H(b)$. But it follows from Lemma \ref{Lem-Tunbounded-Clark} $(ii)$ that this is equivalent to the fact that $\mu_\xi$ is absolutely continuous with respect to $m$.

$(iii)$ Assume now that $|c|>1$. According to Lemma~\ref{outer-is-necessary}, we may assume that $1-cb$ is outer. According to Corollary~\ref{nvlExpCyclShift}, we have to prove that 
\[
\Span_{\mathcal H(b)}((1-cb)k_\lambda:|\lambda|<\norm{S_b}^{-1})=\mathcal H(b).
\]
For that purpose, let $f\in\mathcal H(b)$ and assume that $f$ is orthogonal to $(1-cb)k_\lambda$ in $\mathcal H(b)$, for every $|\lambda|<\norm{S_b}^{-1}$. Then
\[
0=\scal{f}{(1-cb)k_\lambda}_b=\scal{f}{k_\lambda}_b-\overline{c}\scal{f}{bk_\lambda}_b,
\]
and, according to \eqref{DecompScalfNoyrep}, we deduce that 
\[
0=f(\lambda)+\frac{b(\lambda)}{a(\lambda)}f^{+}(\lambda)-\overline{c}\frac{f^+(\lambda)}{a(\lambda)},
\]
for every $|\lambda|<\norm{S_b}^{-1}$. Hence, $af+bf^+-\overline{c}f^+$ vanishes on the disc of center $0$ and radius $\norm{S_b}^{-1}$, and by analytic continuation, it follows that 
$af+bf^+-\overline{c}f^+\equiv 0$ on $\D$. Taking radial limits, we then get 
\[
af+bf^+-\overline{c}f^+=0\qquad \mbox{a.e. on }\T.
\]
Multiply by $\overline{a}$ and use the fact that $|a|^2+|b|^2=1$ a.e. on $\T$ to get 
\[
(1-|b|^2)f+\overline{a}(b-\overline{c})f^+=0\qquad \mbox{a.e. on }\T.
\]
If $g$ is the $H^2$ function satisfying \eqref{eq223:carac-membership-Hb}, we can rewrite the last identity as
\[
f(1-\overline{cb})=(b-\overline{c})\overline{zg}\qquad \mbox{a.e. on }\T,
\]
whence 
\[
\frac{f}{b-\overline{c}}=\frac{\overline{zg}}{1-\overline{cb}}\qquad \mbox{a.e. on }\T.
\]
Observe now that $\inf_\D|b-\overline{c}|\geq |c|-1>0$, whence $\frac{f}{b-\overline{c}}\in H^2$. Moreover, using that $1-cb$ is outer, we have $\frac{\overline{f}}{\overline{b}-c}=\frac{zg}{1-cb}\in L^2(\T)\cap\mathcal N^+=H^2$, and since it vanishes at $0$, $\frac{\overline{f}}{\overline{b}-c}\in H_0^2$. Hence $\frac{f}{b-\overline{c}}\in H^2\cap\overline{H_0^2}=\{0\}$. This concludes the proof of the fact that if $|c|>1$ and $1-cb$ is outer, then $1-cb$ is cyclic for $S_b$. 
\end{proof}

%

Aleman and Malmann have shown in \cite[Theorem 5.11]{InvSubHbForm} that if $\mathcal E$ is a closed invariant subspace of $S_b$, then $\mbox{dim}(\mathcal E\ominus S_b\mathcal E)=1$ and if $\psi\in\mathcal E\ominus S_b\mathcal E$, $\psi\neq 0$,  then 
\begin{equation}\label{thm-Aleman-Malmann}
\mathcal E=\mathcal E_\psi:=\left \{ g \in \mathcal{H}(b) : \frac{g}{\psi} \in H^2, \frac{g}{\psi}\psi^+ \in H^2 \right\}.
\end{equation}
Using this description, we can give a nice characterization of spaces $\mathcal H(b)$ in which the sequence $(fk_{\lambda_n})_{n\geq 1}$ is complete for every outer function $f\in\mathcal H(b)$. 

\begin{theo}\label{EspImadenseSpanHB}
     Let $(\lambda_n)_{n \ge 1} \subset \D$ satisfying $\sum_{n=1}^\infty (1 - \abs{\lambda_n}) =\infty$. Then the following assertions are equivalent:
     \begin{enumerate}
         \item[(i)] for every outer function $f \in \mathcal{H}(b)$, we have 
         \[ \Span_{\mathcal H(b)}(fk_{\lambda_n} : n \ge 1) = \mathcal{H}(b);\]
         \item[(ii)] the subspace $aH^2$ is dense in $\mathcal{H}(b)$.
     \end{enumerate}
\end{theo}

\begin{proof}
$(ii)\implies (i)$: Let $f$ be an outer function in $\mathcal H(b)$. Without loss of generality, we may assume that for all $n \ge 1$, $\lambda_n \neq 0$. Since $aH^2$ is assumed to be dense in $\mathcal H(b)$, it is sufficient to show that 
\begin{equation}\label{eq:3433SFSFSFS}
aH^2\subset\Span_{\mathcal H(b)}(fk_{\lambda_n}:n\geq 1),
\end{equation}
 to get that $(fk_{\lambda_n})_{n\geq 1}$ is complete in $\mathcal H(b)$. Observe that, according to Lemma~\ref{SpanLatSb}, the closed subspace $\Span_{\mathcal H(b)}(fk_{\lambda_n}:n\geq 1)$ is invariant with respect to $S_b$. Thus, it follows from \eqref{thm-Aleman-Malmann} that 
\[
\Span_{\mathcal H(b)}(fk_{\lambda_n}:n\geq 1)=\mathcal E_\psi=\left \{ g \in \mathcal{H}(b) : \frac{g}{\psi} \in H^2, \frac{g}{\psi}\psi^+ \in H^2 \right\},
\]
for some $\psi\in\mathcal H(b)$. Moreover, since $a \mathcal{H}(b) \subset a H^2 \subset \mathcal{H}(b)$, the function $a$ is a multiplier of $\mathcal H(b)$, and then it is easy to see that $a\mathcal E_\psi\subset\mathcal E_\psi$. In other words, we have 
\begin{equation}\label{eq:invariant-a}
a \Span_{\mathcal H(b)}(f k_{\lambda_n} : n \ge 1) \subset \Span_{\mathcal H(b)}(f k_{\lambda_n} : n \ge 1).
\end{equation}
In order to prove \eqref{eq:3433SFSFSFS}, let $h\in H^2$ and $\varepsilon>0$. According to Corollary~\ref{cor:completeness-Hardy}, there exists $q\in\bigvee(k_{\lambda_n}:n\geq 1)$ such that 
\begin{equation}\label{eq:sdfdsfsddfsd879832883}
\|h-qf\|_2\leq \frac{\varepsilon}{2}.
\end{equation}
Using \eqref{eq:invariant-a}, we see that $aqf\in\Span_{\mathcal H(b)}(fk_{\lambda_n}:n\geq 1)$, and so there exists $p\in\bigvee(k_{\lambda_n}:n\geq 1)$ such that
\begin{equation}\label{eq2:sdfdsfsddfsd879832883}
\|aqf-pf\|_b\leq \frac{\varepsilon}{2}.
\end{equation}
Thus, it follows from \eqref{eq:Macontractivelycontained} that 
\[
\|ah-pf\|_b\leq \|ah-aqf\|_b+\|aqf-pf\|_b\leq \|h-qf\|_2+\|aqf-pf\|_b,
\]
which gives by \eqref{eq:sdfdsfsddfsd879832883} and \eqref{eq2:sdfdsfsddfsd879832883},
\[
\|ah-pf\|_b\leq \frac{\varepsilon}{2}+\frac{\varepsilon}{2}=\varepsilon.
\]
Hence for every $h\in H^2$, $ah\in\Span_{\mathcal H(b)}(fk_{\lambda_n}:n\geq 1)$, which proves \eqref{eq:3433SFSFSFS} and then $(i)$.

$(i)\implies (ii)$: Note that 
\[
\bigvee(ak_{\lambda_n}:n\geq 1)\subset aH^2\subset \mathcal H(b).
\]
Since $a$ is outer, it follows from $(i)$ that $\bigvee(ak_{\lambda_n}:n\geq 1)$ is dense in $\mathcal H(b)$, which implies that $aH^2$ is dense in $\mathcal H(b)$. 
\end{proof}

%
%
\begin{coro}
 Let $(\lambda_n)_{n \ge 1} \subset \D$ satisfying $\sum_{n=1}^\infty (1 - \abs{\lambda_n}) = \infty$. Assume that $b(0)=0$ and let $\alpha\in\T$ such that $\mu_\alpha$ is absolutely continuous with respect to $m$. Then the following assertions are equivalent:
 \begin{enumerate}
 \item[(i)] $\Span_{\mathcal H(b)}(ak_{\lambda_n}:n\geq 1)=\mathcal H(b)$;
  \item[(ii)] the function $F_\alpha^2$ is a rigid function.
 \end{enumerate}
\end{coro}

\begin{proof}
$(i)\implies (ii)$: Since 
\[
\bigvee(ak_{\lambda_n}:n\geq 1)\subset aH^2\subset \mathcal H(b),
\]
it follows from $(i)$ that $aH^2$ is dense in $\mathcal H(b)$. Thus, according to Theorem~\ref{thm-densite-ma-rigid}, $F_\alpha^2$ is a rigid function.

$(ii)\implies (i)$: Assume now that $F_\alpha^2$ is rigid. Hence, using one more time Theorem~\ref{thm-densite-ma-rigid}, the subspace $aH^2$ is dense in $\mathcal H(b)$. Then Theorem \ref{EspImadenseSpanHB} implies that $\Span_{\mathcal H(b)}(ak_{\lambda_n}:n\geq 1)=\mathcal H(b)$. 
\end{proof}

%
%

\begin{rem}
After completing this research work, we learned of the existence of Bergman's article \cite{CyclBG}. We thank him for sharing his preprint with us. Note that his results on the cyclic vectors for $S_b$ are much more general than those obtained in our paper. However, our method is (in a sense) more direct and easier and may therefore be of interest. On the other hand, the link with the completeness problem is also interesting. Finally, note that by combining results of  \cite{CyclBG} and Corollary~\ref{CaraCyclSbSpanHb}, we can obtain some generalizations of Theorem~\ref{MHypfcyclSb}, and get more general sufficient conditions for the completeness of $(fk_{\lambda_n})_{n\geq 1}$ in $\mathcal H(b)$. Since it is just an application of results contained in \cite{CyclBG}, we leave the details to the reader.

%
%
\end{rem}

\bibliographystyle{plain}

\bibliography{sources.bib}

\end{document}